\documentclass [11pt]{article}

\usepackage{hyperref}
\usepackage{amsmath, amssymb, amsfonts, amsthm, enumerate}
\date{}
\newtheorem{theo}{Theorem}[section]

\newtheorem{lemma}[theo]{Lemma}

\theoremstyle{definition}

\newcommand{\mb}[1]{\mathbb{#1}}

\newcommand{\bangle}[1]{\left\langle #1 \right\rangle}

\newcommand{\sub}{\subseteq}

\newcommand{\Lra}{\Leftrightarrow}

\newcommand{\bB}{\beta}

\newcommand{\lL}{\lambda}

\def\qed{\hfill $\Box$}

\title{\vspace{-1.1cm} Bounds for spherical codes}

\author{Peter Keevash
\thanks{Mathematical Institute, University of Oxford, Oxford, UK. 
{\tt keevash@maths.ox.ac.uk}.
Research supported in part by ERC Consolidator Grant 647678.}
\and Benny Sudakov
\thanks{Department of Mathematics, ETH, 8092 Zurich, Switzerland. {\tt benjamin.sudakov@math.ethz.ch}. 
Research supported in part by SNSF grant 200021-149111.}
}

\begin{document}

\maketitle

\begin{abstract}
A set $C$ of unit vectors in $\mb{R}^d$ is called an $L$-spherical code 
if $x \cdot y \in L$ for any distinct $x,y$ in $C$.
Spherical codes have been extensively studied 
since their introduction in the 1970's by Delsarte, Goethals and Seidel.
In this note we prove a conjecture of Bukh on the maximum size of spherical codes. 
In particular, we show that for any set of $k$ fixed angles,
one can choose at most $O(d^k)$ lines in $\mb{R}^d$
such that any pair of them forms one of these angles.
\end{abstract}

\section{Introduction}
A set of lines in $\mb{R}^d$ is called equiangular
if the angles between any two of them are the same.
The problem of estimating the size of the maximum family of equiangular lines 
has had a long history since being posed by van Lint and Seidel \cite{vS} in 1966. 
Soon after that, Delsarte, Goethals and Seidel \cite{DGS2} 
showed that for any set of $k$ angles, 
one can choose at most $O(d^{2k})$ lines in $\mb{R}^d$ 
such that every pair of them forms one of these angles.
By choosing a unit direction vector on every line, 
the problem of lines with few angles has the following equivalent formulation. 
Given a set $L=\{a_1, \ldots, a_k\} \sub [-1,1]$.
find the largest set $C$ of unit vectors in $\mb{R}^d$ 
such that $x \cdot y \in L$ for any distinct $x, y \in C$. 
(Here $x \cdot y =\sum_i x_iy_i$ is the standard inner product.) 
Hence the problem of lines with few angles is a special case 
of a more general question which we will discuss next.

Suppose $C$ is a set of unit vectors in $\mb{R}^d$ and $L \sub [-1,1]$.
We say $C$ is an $L$-spherical code if $x \cdot y \in L$ for any distinct $x,y$ in $C$.
We will prove the following theorem on the maximum size of certain spherical codes, 
which was conjectured by Bukh \cite[Conjecture 9]{B1}.

\begin{theo} \label{main}
For any $k \ge 0$ there is a function $f_k: (0,1) \to \mb{R}$ such that
if $0<\bB<1$, $A \sub \mb{R}$ with $|A|=k$ and $C$ is an $L$-spherical code 
in $\mb{R}^d$ with $L =  [-1,-\bB] \cup A$ then $|C| \le f_k(\bB) d^k$.
\end{theo}

In particular, for any set of $k$ fixed angles,
one can choose at most $O(d^k)$ lines in $\mb{R}^d$
such that any pair of them forms one of these angles.
This substantially improves the above-mentioned bound 
of Delsarte, Goethals and Seidel \cite{DGS2}, 
in the case when the angles are fixed, i.e.\ do not depend on the dimension $d$.
The case $k=1$ was proved by Bukh \cite[Theorem 1]{B1}, 
who gave the first linear bound for the equiangular lines problem.
One should note that the assumption that the angles are fixed is important. 
Otherwise, for example when $k=1$, the linear upper bound is no longer valid,
as there are constructions of quadratically many equiangular lines in $\mb{R}^d$ 
(see \cite{D, GKMS, JW}).

\section{Lemmas}
In this section we present several lemmas which we will use in the proof of our main theorem.
We start by recalling some well-known results.
First we need the following bound on $L$-spherical codes, first proved in slightly stronger form by Delsarte, Goethals and 
Seidel \cite{DGS}. At around the same time, Koornwinder \cite{Kor} gave a short elegant proof using linear algebra (see also \cite[Lemma 10]{B2}).

\begin{lemma} \label{alg}
If $L \sub \mb{R}$ with $|L|=k$ and $C$ is an $L$-spherical code in $\mb{R}^d$
then $|C| \le \tbinom{d+k}{k}$.
\end{lemma}

Next we need a well-known variant of Ramsey's theorem, whose short proof we include for the convenience of the reader.
Let $K_n$ denote the complete graph on $n$ vertices. Given an edge-colouring of $K_n$, we call an ordered pair $(X,Y)$ of 
disjoint subsets of vertices monochromatic if all edges in $X \cup Y$ incident to a vertex in $X$ have the same colour.

\begin{lemma} \label{ramsey} 
Let $k, t, m, n$ be non-negative integers satisfying $n>k^{kt}m$ 
and let $f:E(K_n) \to [k]$ be an edge $k$-colouring of $K_n$.
Then there is a monochromatic pair $(X,Y)$ such that $|X|=t$ and $|Y|=m$.
\end{lemma}

\begin{proof}
Consider a family of $kt$ vertices $v_1,\dots,v_{kt}$ and sets $Y_1, \ldots, Y_{kt}$ constructed as follows.
Fix $v_1$ arbitrarily and let $c(1) \in [k]$ be a majority colour among the edges $(v_1,u)$. Set 
$Y_1=\{u: f(v_1,u)=c(1)\}$. By the pigeonhole principle, $|Y_1| \geq \lceil(n-1)/k\rceil \geq k^{kt-1}m$.
In general, we fix any $v_{i+1}$ in $Y_i$, let $c(i+1) \in [k]$ be a majority colour 
among the edges $(v_{j+1},u)$ with $u \in Y_i$, and let $Y_{i+1}=\{u \in Y_i: f(v_{i+1},u)=c(i+1)\}$. 
Then  $|Y_{i+1}| \geq \lceil(|Y_i|-1)/k\rceil \geq k^{kt-i-1}m$, and for every 
$1 \leq j \leq i$ the edges from $v_j$ to all vertices in $Y_{i+1}$ have colour $c(j)$.   
Since we  have only $k$ colours, there is a colour $c \in [k]$ and $S \sub [kt]$ with $|S|=t$
so that $c(j) = c$ for all $j \in S$. Then $X=\{v_j: j \in S\}$ and $Y=Y_{kt}$ form 
a monochromatic pair of colour $c$, satisfying the assertion of the lemma.
\end{proof}

\medskip

The following lemma is also well-known.

\begin{lemma} \label{neg}
If $L=[-1,-\bB]$ and $C$ is a $L$-spherical code then $|C| \le \bB^{-1}+1$. 
\end{lemma}

\begin{proof}
Let $v=\sum_{x \in C} x$. Then, by definition of $L$-spherical code,
$$ 0 \le \| v \|^2 = \sum_{x \in C} \|x\|^2 + \sum_{x\not =x' \in C}x \cdot x'\leq |C| - |C|(|C|-1) \bB=|C|\big(1-(|C|-1) \bB\big).$$
Therefore $1-(|C|-1) \bB \geq 0$, implying $|C| \le \bB^{-1}+1$.
\end{proof}

\medskip

We will also need the following simple corollary of 
Tur\'an's theorem, which can be obtained by greedily deleting 
vertices together with their neighbourhoods.

\begin{lemma} \label{turan}
Every graph on $n$ vertices with maximum degree $\Delta$ 
contains an independent set of size at least $\frac{n}{\Delta+1}$.
\end{lemma}

\medskip

In the remainder of this section 
we will introduce our new tools for bounding spherical codes.
Suppose $x \in \mb{R}^d$ and $U$ is a subspace of $\mb{R}^d$.
We write $x_U$ for the projection of $x$ on $U$.
Let $U^\perp$ be the orthogonal complement of $U$.
Note that $x = x_U + x_{U^\perp}$. If $x_{U^\perp} \ne 0$
we write $p_U(x) = \| x_{U^\perp} \|^{-1} x_{U^\perp}$ for the normalized projection of 
$x$ on $U^\perp$. So $\|p_U(x)\|=1$. If $U=\bangle{Y}$ is spanned by the set of vectors $Y$ we also use $p_Y(x)$ to denote $p_U(x)$.

\begin{lemma} \label{project}
Suppose $\|x_1\|=\|x_2\|=\|y\|=1$ and each $x_i \cdot y = c_i$ with $|c_i|<1$.
Then each $p_y(x_i) = \frac{x_i-c_i y}{\sqrt{1-c_i^2}}$ and 
$p_y(x_1) \cdot p_y(x_2) = \frac{x_1 \cdot x_2 - c_1 c_2}{\sqrt{(1-c_1^2)(1-c_2^2)}}$.
\end{lemma}

\begin{proof}
The projection of $x_i$ on $y$ is $c_i y$, so the projection of $x_i$ on $y^\perp$  is $x_i-c_i y$.
As $(x_i-c_i y) \cdot (x_i-c_i y) = 1-c_i^2$ and $(x_1-c_1 y) \cdot (x_2-c_2 y) = x_1 \cdot x_2 -c_1c_2$
the lemma follows. 
\end{proof}

\medskip

Given a subspace $U$ we can calculate $p_U(x)$ using 
the following version of the Gram-Schmidt algorithm.
Suppose that $\{y_1,\dots,y_k\}$ is a basis for $U$. Write $y_{k+1}=x$.
Define vectors $y^i_j$ by $y^0_j=y_j$ for $j \in [k+1]$
and $y^i_j=p_{y^{i-1}_i}(y^{i-1}_j)$ for $1 \le i < j \le k+1$. It is easy to check by induction that for every $j$ 
the vectors $y_1^0, y_2^1, \ldots, y_j^{j-1}$ are orthogonal. Also $y_j^{j-1}$ is a unit vector for $j>1$.
Therefore $p_U(x) = y^k_{k+1}$. 

\begin{lemma} \label{project+}
Suppose $X \cup Y$ is a set of unit vectors in $\mb{R}^d$
such that $x \cdot y = y \cdot y' = c$ with $|c|<1$ 
for all $x \in X$ and distinct $y,y'$ in $Y$. 
Let $U = \bangle{Y}$ and $k=|Y|$. 
Then for any $x,x'$ in $X$ we have 
$p_U(x) \cdot p_U(x') = g^c_k(x \cdot x')$, where 
\[g^c_k(a) := 1 - (1-c)^{-1} ( 1 - (c^{-1}+k)^{-1} )(1-a)
= (1-c)^{-1} [ a-c + (c^{-1}+k)^{-1} (1-a) ].\]
\end{lemma}

\vspace{0.1cm}
\noindent
{\bf Remark.} \, Note that $g^c_0(a) = a$, $g^c_k(c)=(c^{-1}+k)^{-1}$ and $g^c_k(a)$ is decreasing in $k$. Also
$g^c_k \to \frac{a-c}{1-c}$ when $k$ tends to infinity.

\medskip

\begin{proof}
We write $Y = \{y_1,\dots,y_k\}$, $y_{k+1}=x$, $y_{k+2}=x'$
and calculate $p_U(x) = y^k_{k+1}$ and $p_U(x') = y^k_{k+2}$
using the algorithm and notation introduced before the lemma.
It is easy to see that vectors in $Y$ are linearly independent, 
since the matrix of pairwise inner products of these vectors has full rank.
Let $c_i^{-1}=i+c^{-1}$. We show by induction for $0 \le i \le k$ 
that $y^i_j \cdot y^i_{j'} = c_i$ for all distinct $j,j' > i$,
with the possible exception of $\{j,j'\}=\{k+1,k+2\}$.
Indeed, this holds by hypothesis when $i=0$. 
When $0<i \leq k$, by induction $y^{i-1}_i\cdot y^{i-1}_j=y^{i-1}_i\cdot y^{i-1}_{j'}=c_{i-1}$. 
Therefore by Lemma \ref{project} 
\begin{equation}
\label{induction-step}
y^i_j \cdot y^i_{j'} = p_{y^{i-1}_i}(y^{i-1}_j) \cdot p_{y^{i-1}_i}(y^{i-1}_{j'})
= (1-c_{i-1}^2)^{-1} (y^{i-1}_j \cdot y^{i-1}_{j'} - c_{i-1}^2).
\end{equation}
If $\{j,j'\}\not =\{k+1,k+2\}$, then $y^{i-1}_j \cdot y^{i-1}_{j'}=c_{i-1}$ as well. The induction step follows, as
\[ (y^i_j \cdot y^i_{j'})^{-1} = (1-c_{i-1}^2)(c_{i-1} - c_{i-1}^2)^{-1} = 1+c_{i-1}^{-1} = i+c^{-1}=c_i^{-1} . \]
Writing $r_i = y^i_{k+1} \cdot y^i_{k+2} - 1$ we have $r_{i+1} = (1-c_i^2)^{-1} r_i$ by (\ref{induction-step}),
so \[p_U(x) \cdot p_U(x') = 1 + r_k = 1 - \lL (1 - x \cdot x'),\]
where $\lL = \prod_{i=0}^{k-1} (1-c_i^2)^{-1}$. To compute $\lambda$ consider the case $x \cdot x' = c$. 
Then by the above discussion 
$1 - \lL (1-c) = p_U(x) \cdot p_U(x') = c_k = (c^{-1}+k)^{-1}$,
so $\lL = (1-c)^{-1} ( 1 - (c^{-1}+k)^{-1} )$.
\end{proof}

\section{Proof of the main result}
In this section we prove Theorem \ref{main}. 
We argue by induction on $k$. The base case is $k=0$, when $L =  [-1,-\bB]$, 
and we can take $f_0(\bB) = \bB^{-1}+1$ by Lemma \ref{neg}. Henceforth we suppose $k>0$.
We can assume $d \geq d_0= (2k)^{2k\bB^{-1}}$. Indeed, if we can prove the theorem under this assumption, 
then for $d<d_0$ we can use the upper bound for $\mb{R}^{d_0}$ (since it contains $\mb{R}^d$). 
Then we can deduce the bound for the general case by 
multiplying $f_k(\bB)$ (obtained for the case $d \geq d_0$) 
by a factor $d_0^k=(2k)^{2k^2 \bB^{-1}}$.

Suppose $C = \{x_1,\dots,x_n\}$ is an $L$-spherical code in $\mb{R}^d$,
where $L =  [-1,-\bB] \cup \{a_1,\dots,a_k\}$, with $a_1<\dots<a_k$.
We define graphs $G_0,\dots,G_k$ on $[n]$
where $(i,j) \in G_\ell \Lra x_i \cdot x_j = a_\ell$ for $\ell \in [k]$
and $(i,j) \in G_0 \Lra x_i \cdot x_j \in [-1,-\bB]$.

Consider the case $a_k < \bB^2/2$.
We claim that $G_0$ has maximum degree $\Delta \leq 2\bB^{-2}+1$. 
Indeed, consider $y \in [n]$ and $J \sub [n]$
such that $(y,j) \in G_0$ for all $j \in J$.
For any $j,j'$ in $J$ we have $x_y\cdot x_j, x_y\cdot x_{j'} \leq -\bB$. 
Hence, by Lemma \ref{project} we have
$$p_{x_y}(x_j) \cdot p_{x_y}(x_{j'})
= \frac{x_j\cdot x_{j'}-(x_y\cdot x_j)(x_y\cdot x_{j'})}{\sqrt{1-(x_y\cdot x_j)^2}\sqrt{1-(x_y\cdot x_{j'})^2}} 
\le \frac{a_k - \bB^2}{{\sqrt{1-(x_y\cdot x_j)^2}\sqrt{1-(x_y\cdot x_{j'})^2}}} < -\bB^2/2\,.$$
Thus $|J| \le 2\bB^{-2}+1$ by Lemma \ref{neg}, as claimed.
By Lemma \ref{turan}, $G_0$ has an independent set $S$ of size $n/(2\bB^{-2}+2)$.
Then $\{x_j: j \in S\}$ is an $\{a_1,\dots,a_k\}$-spherical code,
so $|S| \le d^k+1 \leq 2d^k$ by Lemma \ref{alg}. Choosing $f_k(\bB) > 4\bB^{-2}+4$,
we see that the theorem holds in this case.
Henceforth we suppose $a_k \ge \bB^2/2$.

Next consider the case that there is $\ell\geq 2$ such that $a_{\ell-1} < a_\ell^2/2$.
Choosing the maximum such $\ell$ we have 
\begin{equation}
\label{case2}
a_\ell^2/2=2(a_\ell/2)^2 \geq 2(a_{\ell+1}/2)^4 \geq \ldots \geq 2(a_k/2)^{2^{k-\ell+1}} \ge \bB' := (\bB/2)^{2^k}.
\end{equation}
Note that by induction $\cup_{i=0}^{\ell-1} G_i$ contains no clique of order $f_{\ell-1}(\bB)d^{\ell-1}$,
so by Lemma \ref{turan} its complement has maximum degree at least $n'=n/(2f_{\ell-1}(\bB)d^{\ell-1})$.
Consider $y \in [n]$ and $J \sub [n]$ with $|J|=n'$
such that $(y,j) \notin \cup_{i=0}^{\ell-1} G_i$ for all $j \in J$. 
By the pigeonhole principle, there is a subset $J' \subset J$ of size at least
$|J'|\geq |J|/k$  and an index $\ell \leq s\leq k$ such that $(y,j) \in  G_s$ for all $j \in J'$.
For any $(j,j') \in \cup_{i=0}^{\ell-1} G_i[J']$, by Lemma \ref{project} we have
$$p_{x_y}(x_j) \cdot p_{x_y}(x_{j'}) =\frac{x_j\cdot x_{j'}-a_s^2}{1-a_s^2}\leq a_\ell^2/2 - a_\ell^2 
< -a_\ell^2/2 \le - \bB'.$$ 
Now $\{p_{x_y}(x_j): j \in J'\}$ is an $L'$-spherical code,
where $L' = [-1,-\bB'] \cup \{a'_\ell,\dots,a'_k\}$,
with $a'_i = \frac{a_i - a_s^2}{1-a_s^2}$ for $i \ge \ell$.
By induction hypothesis, we have $|J'| \le f_{k-\ell+1}(\bB') d^{k-\ell+1}$,
so choosing $f_k(\bB)>2kf_{\ell-1}(\bB)f_{k-\ell+1}(\bB')$ the theorem holds in this case.

Now suppose that there is no $\ell>1$ such that $a_{\ell-1} < a_\ell^2/2$.
We must have $a_1>0$. Let $t=\lceil 1/\bB' \rceil$.
We apply Lemma \ref{ramsey} to find an index $r$ and a disjoint pair of sets $(T,M)$ 
with $|T|=t$ and $|M|=m\geq  (k+1)^{-(k+1)t} n$, 
such that all vertices in $T$ are adjacent to each other 
and to all vertices in $M$ by edges of $G_r$. 
Note that $r>0$, as $G_0$ has no clique of size $t$ by Lemma \ref{neg}. 
For $j \in M$ we write $x'_j = p_T(x_j)$.
By Lemma \ref{project+}, for any $(j,j') \in G_i[M]$ with $i \ge 1$
we have $x'_j \cdot x'_{j'} = a'_i := g^{a_r}_t(a_i)$.
Also, if $(j,j') \in G_0[N]$ we have 
$x'_j \cdot x'_{j'} = g^{a_r}_t(x_j \cdot x_{j'}) \le g^{a_r}_t(-\bB) \le -\bB$. 
Thus $\{x'_j: j \in M\}$ is an $L'$-spherical code in $\mb{R}^{d-t}$,
where $L' = [-1,\bB] \cup \{a'_1,\dots,a'_k\}$.

We can assume $a'_k \ge \bB^2/2$, otherwise choosing
$f_k(\bB) > (k+1)^{(k+1)t} (4\bB^{-2}+4)$ we are done 
by the first case considered above. Since $a'_r=(a_r^{-1}+t)^{-1}<\bB'$,
the computation in (\ref{case2}) implies that
there is $\ell>1$ such that $a_{\ell-1} < a_\ell^2/2$.
Choosing $f_k(\bB) > (k+1)^{(k+1)t}2kf_{\ell-1}(\bB) f_{k-\ell+1}(\bB')$
we are done by the second case considered above. \qed

\section{Concluding remarks}
One can use our proof to derive an explicit bound for $f_k(\bB)$. Indeed, it can be easily shown 
that it is enough to take $f_k(\bB)$ to be $2^{\bB^{-2^{O(k^2)}}}$. We omit the details, 
as we believe that this bound is very far from optimal. Moreover, one cannot expect a bound better than exponential in $\bB^{-1}$
using our methods or those of Bukh \cite{B1}. On the other hand, we do not know any example ruling out the possibility
that $f_k(\bB)$ could be independent of $\bB$ if $k>0$ and $A$ is fixed (Bukh \cite{B1} also makes this remark for $k=1$). 
One place to look for an improvement is in the application of Ramsey's theorem,
as one would expect much better bounds for Ramsey-type questions for graphs defined by geometric constraints
(see \cite{CFPSS} and its references for examples of this phenomenon).

\end{document}